\documentclass[11pt,oneside,english,reqno]{amsart}
\usepackage[T1]{fontenc}
\usepackage[latin9]{inputenc}
\usepackage{geometry}
\usepackage{babel}
\usepackage{mathrsfs}
\usepackage{amsthm}
\usepackage{amstext}
\usepackage{amssymb}
\usepackage{setspace}
\usepackage{esint}
\PassOptionsToPackage{normalem}{ulem}
\usepackage{ulem}
\usepackage[unicode=true,pdfusetitle,
 bookmarks=true,bookmarksnumbered=false,bookmarksopen=false,
 breaklinks=false,pdfborder={0 0 1},backref=false,colorlinks=false]
 {hyperref}
\hypersetup{
 colorlinks=true,citecolor=blue,linkcolor=blue,linktocpage=true}

\makeatletter
\numberwithin{equation}{section}
\numberwithin{figure}{section}
\numberwithin{table}{section}
\usepackage{enumitem}		
\theoremstyle{plain}
\newtheorem{thm}{\protect\theoremname}[section]
  \theoremstyle{remark}
  \newtheorem{rem}[thm]{\protect\remarkname}
  \theoremstyle{definition}
  \newtheorem{defn}[thm]{\protect\definitionname}
  \theoremstyle{plain}
  \newtheorem{lem}[thm]{\protect\lemmaname}
  \theoremstyle{plain}
  \newtheorem{cor}[thm]{\protect\corollaryname}
  \theoremstyle{plain}
  \newtheorem{prop}[thm]{\protect\propositionname}
  \theoremstyle{remark}
  \newtheorem*{acknowledgement*}{\protect\acknowledgementname}

\allowdisplaybreaks

\providecommand{\MR}[1]{}

\renewcommand{\section}{%
\@startsection{section}{1}%
  \z@{.7\linespacing\@plus\linespacing}{.5\linespacing}%
  {\normalfont\scshape\centering\bfseries}}

\renewcommand{\subsection}{%
\@startsection{subsection}{2}%
  \z@{.5\linespacing\@plus.7\linespacing}{.5\linespacing}%
  {\normalfont\bfseries}}

\renewcommand{\subsubsection}{%
\@startsection{subsubsection}{2}%
  \z@{.5\linespacing\@plus.7\linespacing}{.5\linespacing}%
  {\normalfont\bfseries}}

\AtBeginDocument{
  
}

\makeatother

  \providecommand{\acknowledgementname}{Acknowledgement}
  \providecommand{\corollaryname}{Corollary}
  \providecommand{\definitionname}{Definition}
  \providecommand{\lemmaname}{Lemma}
  \providecommand{\propositionname}{Proposition}
  \providecommand{\remarkname}{Remark}
\providecommand{\theoremname}{Theorem}

\begin{document}

\title{Generalized Gramians: Creating frame vectors in maximal subspaces}

\author{Palle Jorgensen and Feng Tian}

\address{(Palle E.T. Jorgensen) Department of Mathematics, The University
of Iowa, Iowa City, IA 52242-1419, U.S.A. }

\email{palle-jorgensen@uiowa.edu}

\urladdr{http://www.math.uiowa.edu/\textasciitilde{}jorgen/}

\address{(Feng Tian) Department of Mathematics, Trine University, IN 46703,
U.S.A.}

\email{tianf@trine.edu}


\begin{abstract}
A frame is a system of vectors $S$ in Hilbert space $\mathscr{H}$
with properties which allow one to write algorithms for the two operations,
analysis and synthesis, relative to $S$, for all vectors in $\mathscr{H}$;
expressed in norm-convergent series. Traditionally, frame properties
are expressed in terms of an $S$-Gramian, $G_{S}$ (an infinite matrix
with entries equal to the inner product of pairs of vectors in $S$);
but still with strong restrictions on the given system of vectors
in $S$, in order to guarantee frame-bounds. In this paper we remove
these restrictions on $G_{S}$, and we obtain instead direct-integral
analysis/synthesis formulas. We show that, in spectral subspaces of
every finite interval $J$ in the positive half-line, there are associated
standard frames, with frame-bounds equal the endpoints of $J$. Applications
are given to reproducing kernel Hilbert spaces, and to random fields.
\end{abstract}

\maketitle
\tableofcontents{}

\section{Introduction}

Frames are redundant bases which turn out in certain applications
to be more flexible than the better known orthonormal bases (ONBs)
in Hilbert space; see, e.g., \cite{MR2837145,MR3167899,MR2367342,MR2147063}.
The frames allow for more symmetries than ONBs do, especially in the
context of signal analysis, and of wavelet constructions; see, e.g.,
\cite{CoDa93,BDP05,Dutkay_2006}. Since frame bases (although containing
redundancies) still allow for efficient algorithms, they have found
many applications, even in finite dimensions; see, for example, \cite{BeFi03,CaCh03,Chr03,Eld02,FJKO05}.

As is well known, when a vector $f$ in a Hilbert space $\mathscr{H}$
is expanded in an orthonormal basis $B$, there is then automatically
an associated Parseval identity. In physical terms, this identity
typically reflects a \emph{stability} feature of a decomposition based
on the chosen ONB $B$. Specifically, Parseval's identity reflects
a conserved quantity for a problem at hand, for example, energy conservation
in quantum mechanics.

The theory of frames begins with the observation that there are useful
vector systems which are in fact not ONBs but for which a Parseval
formula still holds. In fact, in applications it is important to go
beyond ONBs. While this viewpoint originated in signal processing
(in connection with frequency bands, aliasing, and filters), the subject
of frames appears now to be of independent interest in mathematics. 

Recently the solution to the Kadison-Singer question, and its equivalent
forms, has proved to take a form involving frames in Hilbert space.
While we will not discuss the Kadison-Singer question here, we mention
it as an illustration of the usefulness of frames in problems which
in their original form were not stated in the language of frames.
The Kadison-Singer question is about uniqueness of extensions of pure
states, where a pure state is a quantum theoretical notion. For interested
readers, we refer to \cite{MR3210706,MR2964016,2014arXiv1407.4768C,2013arXiv1306.3969M}.

This entails systems of vectors $\left\{ \varphi_{n}\right\} _{n\in\mathbb{N}}\subset\mathscr{H}$,
called \emph{frames}, where there are two constants $a,b$, s.t. $0<a\leq b<\infty$,
satisfying:
\begin{equation}
a\left\Vert f\right\Vert _{\mathscr{H}}^{2}\leq\sum_{n=1}^{\infty}\left|\left\langle \varphi_{n},f\right\rangle _{\mathscr{H}}\right|^{2}\leq b\left\Vert f\right\Vert _{\mathscr{H}}^{2},\quad\forall f\in\mathscr{H}.\label{eq:d1}
\end{equation}
The two constants $a$ and $b$ in (\ref{eq:d1}) are called \emph{frame
bounds}. (\emph{Parseval frame} means $a=b=1$.)

The purpose of this paper is to associate canonical \emph{projection
valued measures} (PVM) to system of vectors $\left\{ \varphi_{n}\right\} $
which are not frames but such that
\begin{equation}
\sum_{j\in\mathbb{N}}\left|\left\langle \varphi_{j},\varphi_{n}\right\rangle _{\mathscr{H}}\right|^{2}<\infty,\quad\forall n\in\mathbb{N};
\end{equation}
and, from this, then to generate systems of \emph{frames} in suitable
canonical closed subspaces of $\mathscr{H}$, with the subspaces derived
from the PVM, and with a variety of choices of \emph{frame bounds}.

As an application, we give natural and explicit conditions on these
frame systems to generate reproducing kernel Hilbert spaces (RKHS).

The RKHSs have been studied extensively since the pioneering papers
by Aronszajn in the 1940ties, see e.g., \cite{Aro43,Aro48}. They
further play an important role in the theory of partial differential
operators (PDO); for example as Green's functions of second order
elliptic PDOs; see e.g., \cite{Nel57,HKL14}. Other applications include
engineering, physics, machine-learning theory (see \cite{KH11,SZ09,CS02}),
stochastic processes (e.g., Gaussian free fields), numerical analysis,
and more. See, e.g., \cite{AD93,ABDdS93,AD92,AJSV13,AJV14}, and also
\cite{MR2913695,MR2975345,MR3091062,MR3101840}.

\section{Gramians}

Let $\mathscr{H}$ be a Hilbert space, and $\left\{ \varphi_{n}\right\} _{n\in\mathbb{N}}$
an indexed system of vectors in $\mathscr{H}$; then we say that the
$\infty\times\infty$ matrix 
\begin{equation}
G=\left(\left\langle \varphi_{i},\varphi_{j}\right\rangle _{\mathscr{H}}\right)_{i,j=1}^{\infty}\label{eq:m1}
\end{equation}
is a \emph{Gramian}. Gramians are so named since they were first used
in the Gram-Schmidt orthogonalization; but, more recently, they have
played an important role in the analysis of frames. In order to get
frames, it is natural to further assume that 
\begin{equation}
\left\{ \varphi_{n}\right\} _{n\in\mathbb{N}}^{\perp}=\left\{ 0\right\} ,\label{eq:m2}
\end{equation}
i.e., that the vectors span a dense subspace in $\mathscr{H}$.

In order to understand frame properties indexed by such systems, it
is necessary to interpret $G$ in (\ref{eq:m1}) as an operator in
$l^{2}$. This works if we further assume:
\begin{equation}
\text{For \ensuremath{\forall n\in\mathbb{N},\quad}\ensuremath{\sum_{j\in\mathbb{N}}\left|\left\langle \varphi_{j},\varphi_{n}\right\rangle _{\mathscr{H}}\right|^{2}<\infty}.}\label{eq:m3}
\end{equation}
In that case, there is an operator $T_{G}$ as follows:
\begin{equation}
\left(T_{G}c\right)_{n}=\sum_{j=1}^{\infty}\left\langle \varphi_{n},\varphi_{j}\right\rangle _{\mathscr{H}}c_{j},\quad\forall c=\left(c_{j}\right)\in l^{2},\label{eq:m4}
\end{equation}
or $c$ in a dense subspace of $l^{2}$. More precisely, $T_{G}$
defines a Hermitian semi-bounded operator $T_{G}:l^{2}\rightarrow l^{2}$
with dense domain (finite sequences) in $l^{2}$. But in general,
$T_{G}$ (from (\ref{eq:m4})) is unbounded, its domain is only dense
in $l^{2}$; and it may not even be essentially selfadjoint.

By von Neumann's theory \cite{DS88b}, $T_{G}$ is essentially selfadjoint
if and only if the following holds:
\begin{equation}
\left\{ \xi\in l^{2}\:\big|\: T_{G}^{*}\xi=-\xi\right\} =0.\label{eq:m5}
\end{equation}
One checks that (\ref{eq:m5}) holds precisely when 
\begin{equation}
\Big[\sum_{j=1}^{\infty}\left\langle \varphi_{j},\varphi_{n}\right\rangle _{\mathscr{H}}\xi_{j}=-\xi_{n}\Big]\Longrightarrow\xi=0\;\mbox{in}\; l^{2}.\label{eq:m6}
\end{equation}

Using a theorem from \cite{MR0438178,MR507913}, one further checks
that (\ref{eq:m5}), or equivalently, (\ref{eq:m6}), holds if the
Gramian satisfies:

$\exists\left\{ b_{n}\right\} \subset\mathbb{R}_{+}$ such that
\begin{equation}
\begin{split} & \sum_{j\neq n}\left|\left\langle \varphi_{j},\varphi_{n}\right\rangle _{\mathscr{H}}\right|=O\left(b_{n}\right),\;\mbox{and}\\
 & \sum_{n}\frac{1}{\sqrt{b_{n}}}=\infty.
\end{split}
\label{eq:m7}
\end{equation}
 
\begin{rem}
Using Aronszajn's theorem, one shows that the following two (\ref{enu:m1})\&(\ref{enu:m2})
are the ``same'':
\begin{enumerate}
\item \label{enu:m1}$p:\mathbb{Z}_{+}\times\mathbb{Z}_{+}\rightarrow\mathbb{C}$
positive definite function; and
\item \label{enu:m2}a system $(\mathscr{H},\left\{ \varphi_{i}\right\} _{i\in\mathbb{Z}_{+}}\subset\mathscr{H})$
s.t. $\left\{ \varphi_{i}\right\} ^{\perp}=0$, and 
\begin{equation}
p\left(i,j\right)=\left\langle \varphi_{i},\varphi_{j}\right\rangle _{\mathscr{H}},\quad\left(i,j\right)\in\mathbb{Z}_{+}\times\mathbb{Z}_{+}.\label{eq:mm1}
\end{equation}

\end{enumerate}

Let $\mathscr{F}\subset l^{2}$ denote the (dense) subspace of all
\emph{finite} sequence: If 
\begin{equation}
\sum_{j}\left|p\left(i,j\right)\right|^{2}=\sum_{j}\left|\left\langle \varphi_{i},\varphi_{j}\right\rangle \right|^{2}<\infty,\quad\forall i\in\mathbb{Z}_{+},\label{eq:mm2}
\end{equation}
then we get an operator $T_{p}$ in $l^{2}$ as follows:
\begin{equation}
\left(T_{p}c\right)_{i}:=\sum_{j}p\left(i,j\right)c_{j},\quad\forall c\in\mathscr{F}\label{eq:mm3}
\end{equation}
defines a Hermitian semibounded operator (generally unbounded) in
$l^{2}$ with dense domain $dom(T_{p})=\mathscr{F}$, i.e., 
\begin{equation}
\left\langle c,T_{p}c\right\rangle _{l^{2}}\geq0,\quad\forall c\in\mathscr{F}.\label{eq:mm4}
\end{equation}

\end{rem}

\section{Main theorems}
\begin{defn}
Let $\mathscr{H}$ be a Hilbert space, and let $\left\{ \varphi_{n}\right\} _{n=1}^{\infty}\subset\mathscr{H}$
be a system of vectors s.t. $\left\{ \varphi_{n}\right\} ^{\perp}=0$,
i.e., the span is dense in $\mathscr{H}$. Let $a,b\in\mathbb{R}_{+}$,
be given such that $0<a\leq b<\infty$; and let $\mathscr{K}\subset\mathscr{H}$
be a closed subspace. We say that $\mathscr{K}$ is an $\left(a,b\right)$
\emph{frame subspace} iff (Def) 
\begin{equation}
a\left\Vert f\right\Vert ^{2}\leq\sum_{n=1}^{\infty}\left|\left\langle \varphi_{n},f\right\rangle \right|^{2}\leq b\left\Vert f\right\Vert ^{2},\label{eq:f1}
\end{equation}
holds for all vectors $f\in\mathscr{K}$. Here $\left\langle \cdot,\cdot\right\rangle $
denotes the inner product in $\mathscr{H}$, and $\left\Vert f\right\Vert =\left\langle f,f\right\rangle ^{\frac{1}{2}}$
the corresponding norm. \end{defn}
\begin{thm}
\label{thm:mspace}Let $\mathscr{H}$, $\left\{ \varphi_{n}\right\} $
be a pair, $\mathscr{H}$ a Hilbert space, and $\left\{ \varphi_{n}\right\} _{1}^{\infty}\subset\mathscr{H}$,
satisfying $\left\{ \varphi_{n}\right\} ^{\perp}=0$. Assume further
that
\begin{equation}
\sum_{j=1}^{\infty}\left|\left\langle \varphi_{n},\varphi_{j}\right\rangle \right|^{2}<\infty,\quad\forall n\in\mathbb{N},\;\mbox{and}\label{eq:f2a}
\end{equation}
\begin{equation}
\Big[\xi\in l^{2},\;\sum_{j}\left\langle \varphi_{n},\varphi_{j}\right\rangle \xi_{j}=-\xi_{n}\Big]\Longrightarrow\xi=0\;\mbox{in}\; l^{2}.\label{eq:f2b}
\end{equation}
Then, for every $a,b\in\mathbb{R}_{+}$, s.t. $0<a\leq b<\infty$,
there is a unique \uline{maximal} subspace $\mathscr{K}=\mathscr{H}\left(a,b\right)$
which satisfies the frame condition (\ref{eq:f1}).\end{thm}
\begin{proof}
We begin with a specification of a pair of dense subspaces in $\mathscr{H}$,
and in $l^{2}$, $\mathscr{D}_{\mathscr{H}}$, and $\mathscr{D}_{l^{2}}$:
\begin{itemize}
\item[] $\mathscr{D}_{\mathscr{H}}=span\left\{ \varphi_{n}\right\} _{n\in\mathbb{N}}$,
i.e., all finite linear combination of these vectors, and 
\item[] $\mathscr{D}_{l^{2}}=span\left\{ \varepsilon_{n}\right\} $, where
$\varepsilon_{n}\left(j\right)=\delta_{nj}$; i.e., $\mathscr{D}_{l^{2}}$
consists of all finitely supported sequences.
\end{itemize}

We introduce two densely defined operators
\begin{equation}
\mathscr{H}\xrightarrow{\; L\;}l^{2},\quad\mbox{and}\quad l^{2}\xrightarrow{\; M\;}\mathscr{H}
\end{equation}
as follows: Set the respective domains, 
\begin{equation}
dom\left(L\right)=\mathscr{D}_{\mathscr{H}},\quad\mbox{and}\quad dom\left(M\right)=\mathscr{D}_{l^{2}},
\end{equation}
and set 
\begin{eqnarray}
Lf & = & \left(\left\langle \varphi_{n},f\right\rangle \right)_{n},\quad\forall f\in\mathscr{D}_{\mathscr{H}};\label{eq:f3}\\
M\xi & = & \sum_{n=1}^{\infty}\xi_{n}\varphi_{n},\quad\forall\xi\in\mathscr{D}_{l^{2}}.\label{eq:f4}
\end{eqnarray}
From the definition of these two operators, we get the following symmetry:
\begin{equation}
\left\langle Lf,\xi\right\rangle _{l^{2}}=\left\langle f,M\xi\right\rangle _{\mathscr{H}},\quad\forall f\in\mathscr{D}_{\mathscr{H}},\:\forall\xi\in\mathscr{D}_{l^{2}},\label{eq:f5}
\end{equation}
where the respective inner products are indicated with subscripts,
$\left\langle \cdot,\cdot\right\rangle _{\mathscr{H}}$ and $\left\langle \cdot,\cdot\right\rangle _{l^{2}}$.

If $L^{*}$ and $M^{*}$ denote the respective adjoint operators,
then 
\begin{equation}
l^{2}\xrightarrow{\; L^{*}\;}\mathscr{H},\quad\mbox{and}\quad\mathscr{H}\xrightarrow{\; M^{*}\;}l^{2}.\label{eq:f6}
\end{equation}

We will need the following:\end{proof}
\begin{lem}
\label{lem:LM}We have 
\begin{equation}
L\subseteq M^{*},\quad\mbox{and}\quad M\subseteq L^{*},\label{eq:f7}
\end{equation}
where the containments in (\ref{eq:f7}) refer to the respective graphs,
i.e., (\ref{eq:f7}) states that 
\begin{equation}
G\left(L\right)\subseteq G\left(M^{*}\right),\quad\mbox{and}\quad G\left(M\right)\subseteq G\left(L^{*}\right).\label{eq:f8}
\end{equation}
By the graph of a linear operator $T:\mathscr{H}_{1}\rightarrow\mathscr{H}_{2}$
between Hilbert spaces, we mean 
\begin{equation}
G\left(T\right)=\left\{ \begin{pmatrix}h_{1}\\
Th_{1}
\end{pmatrix}\:\Big|\: h_{1}\in dom\left(T\right)\right\} \subset\begin{pmatrix}\underset{\bigoplus}{\mathscr{H}_{1}}\\
\mathscr{H}_{1}
\end{pmatrix}.\label{eq:f9}
\end{equation}
\end{lem}
\begin{proof}
The conclusion in the lemma is immediate from (\ref{eq:f5}).\end{proof}
\begin{cor}
Both of the operators $L$ and $M$ in (\ref{eq:f3})-(\ref{eq:f4})
are \uline{closable}, i.e., the closure of $G\left(L\right)$ is
a densely defined closed operator $\mathscr{H}\xrightarrow{\;\overline{L}\;}l^{2}$,
and the closure of $G\left(M\right)$ is a densely defined closed
operator $l^{2}\xrightarrow{\;\overline{M}\;}\mathscr{H}$.\end{cor}
\begin{proof}
Follows from the fact that a densely defined operator $\mathscr{H}_{1}\xrightarrow{\; T\;}\mathscr{H}_{2}$
is closable if and only if the domain of the adjoint $T^{*}$ is dense
in $\mathscr{H}_{2}$. The assertion that both $L^{*}$ and $M^{*}$
have dense domains is immediate from (\ref{eq:f7}).\end{proof}
\begin{lem}
\label{lem:closure}Let $L$ and $M$ be the two operators in Lemma
\ref{lem:LM}, then 
\begin{equation}
\overline{L}=M^{*},\quad\mbox{and}\quad\overline{M}=L^{*}.\label{eq:f10}
\end{equation}
\end{lem}
\begin{proof}
From (\ref{eq:f7}), we automatically get containments:
\begin{equation}
\overline{L}\subseteq M^{*},\quad\mbox{and}\quad\overline{M}\subseteq L^{*},\label{eq:f11}
\end{equation}
but the conclusion (\ref{eq:f10}) in the lemma, is that both containments
(\ref{eq:f11}) are equality. 

Returning to the definitions (\ref{eq:f3}) and (\ref{eq:f4}), we
note that (\ref{eq:f10}) follows if we prove the following two implications:
\begin{equation}
\left\{ \begin{split} & f\in dom(M^{*}),\;\mbox{and}\\
 & \left\langle \varphi_{n},f\right\rangle _{\mathscr{H}}+\sum_{j=1}^{\infty}\left\langle \varphi_{n},\varphi_{j}\right\rangle _{\mathscr{H}}\left(M^{*}f\right)_{j}=0
\end{split}
\right\} \Longrightarrow f=0\;\mbox{in}\;\mathscr{H}.\label{eq:f12}
\end{equation}
And
\begin{equation}
\left\{ \begin{split} & \xi\in dom(L^{*}),\;\mbox{and}\\
 & \sum_{j=1}^{n}\left\langle \varphi_{n},\varphi_{j}\right\rangle _{\mathscr{H}}\xi_{j}=-\xi_{n},\;\forall n\in\mathbb{N}
\end{split}
\right\} \Longrightarrow\xi=0\;\mbox{in}\; l^{2}.\label{eq:f13}
\end{equation}

Using the symmetry (\ref{eq:f11}), we note that either one of the
two implications (\ref{eq:f12}), or (\ref{eq:f13}), implies the
other.

We will prove (\ref{eq:f13}). For this, we make use of assumption
(\ref{eq:f2a}) and (\ref{eq:f2b}) from the premise in Theorem \ref{thm:mspace}.

Here the $\infty\times\infty$ matrix $\left(\left\langle \varphi_{n},\varphi_{j}\right\rangle _{\mathscr{H}}\right)$
defines a Hermitian symmetric operator $T$ with dense domain in $l^{2}$,
\begin{equation}
\left(Tc\right)_{n}:=\sum_{j}\left\langle \varphi_{n},\varphi_{j}\right\rangle _{\mathscr{H}}c_{j},\quad\forall c=\left(c_{j}\right)\in dom\left(T\right).\label{eq:f14}
\end{equation}

Since $T$ is semibounded, i.e., 
\begin{equation}
\left\langle c,Tc\right\rangle _{l^{2}}\geq0,\quad\forall c\in dom\left(T\right),\label{eq:f15}
\end{equation}
it follows from assumption (\ref{eq:f2b}) that $T$ has deficiency
indices $\left(0,0\right)$; and in particular that the implication
(\ref{eq:f13}) holds.

This completes the proof of Lemma \ref{lem:closure}.
\end{proof}

\begin{proof}[Proof of Theorem \ref{thm:mspace}, continued.]
 We conclude from the lemmas that $L_{1}:=\overline{L}$ (closure)
is a closed operator with dense domain in $\mathscr{H}$, and $\mathscr{H}\xrightarrow{\;\overline{L}\;}l^{2}$. 

By a theorem of von Neumann \cite{DS88b}, we further conclude that
the two operators $L^{*}L_{1}$, and $L_{1}L^{*}$, are both selfadjoint
(s.a.); the first in $\mathscr{H}$, and the second in $l^{2}$ (not
merely Hermitian symmetric), but both have deficiency indices $\left(0,0\right)$,
and both are semibounded, i.e., with spectrum in $\left[0,\infty\right]$. 

Since $L^{*}L_{1}$ is s.a. with dense domain in $\mathscr{H}$, it
has a unique spectral resolution in the form of a projection valued
measure $P\left(\cdot\right)$, i.e., 
\begin{enumerate}
\item $P\left(A\right)=P\left(A\right)^{*}=P\left(A\right)^{2}$, $\forall A\in\mathscr{B}:=\mathscr{B}\left([0,\infty)\right)$; 
\item $P\left(A\cap B\right)=P\left(A\right)P\left(B\right)$, $A,B\in\mathscr{B}$; 
\item $P\left(\cdot\right)$ is sigma-additive; 
\item $P\left([0,\infty)\right)=I_{\mathscr{H}}$, the identity operator
in $\mathscr{H}$; and
\item $L^{*}L_{1}=\int_{0}^{\infty}\lambda P\left(d\lambda\right)$. 
\end{enumerate}

More generally, if $\varphi$ is a Borel function, then by functional
calculus the operator $\varphi\left(L^{*}L_{1}\right)$ is: 
\begin{equation}
\varphi\left(L^{*}L_{1}\right)=\int_{0}^{\infty}\varphi\left(\lambda\right)P\left(d\lambda\right);\label{eq:g2}
\end{equation}
and the following are equivalent:
\begin{equation}
f\in dom\left(\varphi\left(L^{*}L_{1}\right)\right)\Longleftrightarrow\int_{0}^{\infty}\left|\varphi\left(\lambda\right)\right|^{2}\left\Vert P\left(d\lambda\right)f\right\Vert _{\mathscr{H}}^{2}<\infty.\label{eq:g3}
\end{equation}

If $\alpha=\left[a,b\right]$, $0<a\leq b<\infty$, i.e., a fixed
compact interval in $\left(0,\infty\right)$, set 
\begin{equation}
\mathscr{H}_{\alpha}=P\left(\alpha\right)\mathscr{H}.\label{eq:g4}
\end{equation}
Setting function $\varphi\left(\lambda\right)=\lambda\chi_{\alpha}\left(\lambda\right)$
in (\ref{eq:g2}), we get
\begin{equation}
a\left\Vert f\right\Vert _{\mathscr{H}}^{2}\leq\left\langle f,L^{*}L_{1}f\right\rangle _{\mathscr{H}}\leq b\left\Vert f\right\Vert _{\mathscr{H}}^{2},\quad\forall f\in\mathscr{H}_{\alpha}.\label{eq:g5}
\end{equation}
Since $\left\langle f,L^{*}L_{1}f\right\rangle _{\mathscr{H}}=\sum_{n=1}^{\infty}\left|\left\langle \varphi_{n},f\right\rangle _{\mathscr{H}}\right|^{2}$,
it follows that (\ref{eq:g5}) is the desired frame-estimate, i.e.,
\begin{equation}
a\left\Vert f\right\Vert _{\mathscr{H}}^{2}\leq\sum_{n=1}^{\infty}\left|\left\langle \varphi_{n},f\right\rangle _{\mathscr{H}}\right|^{2}\leq b\left\Vert f\right\Vert _{\mathscr{H}}^{2},\quad\forall f\in\mathscr{H}_{\alpha}.\label{eq:g6}
\end{equation}

Moreover, if $\mathscr{K}\subset\mathscr{H}$ is a closed subspace
in $\mathscr{H}$ s.t. (\ref{eq:g6}) holds for $\forall f\in\mathscr{K}$,
then it follows from (\ref{eq:g2})-(\ref{eq:g4}) that $\mathscr{K}\subset\mathscr{H}_{\alpha}$,
$\alpha=\left[a,b\right]$. This concludes the proof of the theorem.
\end{proof}
The following result is implied by Lemma \ref{lem:closure}, and is
of independent interest:
\begin{prop}
Let $\left(\mathscr{H},\left\{ \varphi_{n}\right\} _{n\in\mathbb{N}}\right)$
satisfy the conditions (\ref{eq:f2a})-(\ref{eq:f2b}) in Theorem
\ref{thm:mspace}, and let 
\begin{equation}
L_{1}=U\left(L^{*}L_{1}\right)^{\frac{1}{2}}=\left(L_{1}L^{*}\right)^{\frac{1}{2}}U\label{eq:g7}
\end{equation}
be the corresponding polar decomposition of $\mathscr{H}\xrightarrow{\; L_{1}\;}l^{2}$,
then 
\begin{equation}
U=T_{G}^{-\frac{1}{2}}L_{1},\;\mbox{and}\label{eq:g8}
\end{equation}
$U$ is an isometry of $\mathscr{H}$ into $l^{2}$.\end{prop}
\begin{proof}
By the properties of the polar decomposition (\ref{eq:g7}), $U$
is a partial isometry with initial space = $\mathscr{H}\ominus\ker\left(L_{1}\right)$;
so we only need to prove that $\ker\left(L_{1}\right)=0$. But from
(\ref{eq:f3}) we get $\ker\left(L_{1}\right)=\left\{ \varphi_{n}\right\} ^{\perp}=0$,
by (\ref{eq:m2}).\end{proof}
\begin{defn}
Let $\Omega$ be a set, and let $\mathscr{H}$ be a Hilbert space
of functions on $\Omega$. We say that $\mathscr{H}$ is a \emph{reproducing
kernel Hilbert} space (RKHS) iff (Def): $\mathscr{H}\ni f\longmapsto f\left(t\right)\in\mathbb{C}$
is norm-continuous for all $t\in\Omega$; or, equivalent, if for all
$t\in\Omega$, there exists a $K_{t}\in\mathscr{H}$ s.t. 
\begin{equation}
f\left(t\right)=\left\langle K_{t},f\right\rangle _{\mathscr{H}},\quad\forall f\in\mathscr{H}.\label{eq:e0}
\end{equation}
We say that (\ref{eq:e0}) is the reproducing property. (See, e.g.,
\cite{Aro43,Aro48}.)\end{defn}
\begin{thm}
\label{thm:rkhs}Let $\mathscr{H}$ be a Hilbert space, and $\left\{ \varphi_{n}\right\} _{n\in\mathbb{N}}$
be a \uline{frame} with frame bounds $a,b$, where 
\begin{equation}
0<a\leq b<\infty.\label{eq:e1}
\end{equation}
Suppose further that there is a set $\Omega$ s.t. each $\varphi_{n}$
is a function on $\Omega$, and that 
\begin{equation}
\left(\varphi_{n}\left(t\right)\right)_{n\in\mathbb{N}}\in l^{2},\quad\forall t\in\Omega.\label{eq:e2}
\end{equation}
Then $\mathscr{H}$ is a reproducing kernel Hilbert space (RKHS) of
functions on $\Omega$. \end{thm}
\begin{proof}
By (\ref{eq:e1})\&(\ref{eq:e2}) we conclude that
\begin{equation}
\sum_{n\in\mathbb{N}}\varphi_{n}\left(t\right)\varphi_{n}\in\mathscr{H},\label{eq:e3}
\end{equation}
and that 
\begin{equation}
a\sum_{n}\left|\varphi_{n}\left(t\right)\right|^{2}\leq\left\Vert \sum_{n=1}^{\infty}\varphi_{n}\left(t\right)\varphi_{n}\right\Vert _{\mathscr{H}}^{2}\leq b\sum_{n}\left|\varphi_{n}\left(t\right)\right|^{2}\label{eq:e4}
\end{equation}
holds for all $t\in\Omega$.

For $t\in\Omega$, set 
\begin{equation}
l\left(t\right)=\begin{bmatrix}\varphi_{1}\left(t\right)\\
\varphi_{2}\left(t\right)\\
\vdots\\
\vdots
\end{bmatrix}\in l^{2}\label{eq:e5}
\end{equation}
and set 
\begin{eqnarray}
K^{G}\left(s,t\right) & := & \left\langle l\left(s\right),G^{-1}l\left(t\right)\right\rangle _{2}\label{eq:e6}\\
 & = & \underset{\left(n,m\right)\in\mathbb{N}\times\mathbb{N}}{\sum\sum}\overline{\varphi_{n}\left(s\right)}\left(G^{-1}\right)_{nm}\varphi_{m}\left(t\right).\nonumber 
\end{eqnarray}
We claim that $K^{G}\left(\cdot,\cdot\right)$ in (\ref{eq:e6}) turns
$\mathscr{H}$ into a RKHS. 

To see this, set 
\begin{eqnarray}
Lf & = & \left(\left\langle \varphi_{j},f\right\rangle _{\mathscr{H}}\right)_{j\in\mathbb{N}},\;\mbox{and}\label{eq:e7}\\
L^{*}c & = & \sum_{j=1}^{\infty}c_{j}\varphi_{j}\label{eq:e8}
\end{eqnarray}
then by standard frame theory $L^{*}L$ and $LL^{*}$ are both bounded
and selfadjoint. Moreover, $LL^{*}=T_{G}=$ the operator in $l^{2}$
defined from the Gramian matrix 
\begin{equation}
G=\left(\left\langle \varphi_{i},\varphi_{j}\right\rangle _{\mathscr{H}}\right)_{i,j=1}^{\infty}.\label{eq:e9a}
\end{equation}
Since $a,b$ are the frame bounds, we get:
\begin{align}
aI_{\mathscr{H}} & \leq L^{*}L\leq bI_{\mathscr{H}},\;\mbox{and}\label{eq:e9b}\\
aI_{l^{2}} & \leq LL^{*}\leq bI_{l^{2}}\label{eq:e10}
\end{align}
where $I$ denotes the respective identity operators in $\mathscr{H}$,
and in $l^{2}$; and where ``$\leq$'' denotes the natural ordering
of Hermitian (bounded) operators. It further follows that the system,
\begin{equation}
\psi_{n}:=\left(L^{*}L\right)^{-\frac{1}{2}}\varphi_{n},\quad n\in\mathbb{N}
\end{equation}
is a \emph{Parseval} frame in $\mathscr{H}$, i.e., 
\begin{equation}
\left\Vert f\right\Vert _{\mathscr{H}}^{2}=\sum_{n=1}^{\infty}\left|\left\langle \psi_{n},f\right\rangle _{\mathscr{H}}\right|^{2}\label{eq:e12}
\end{equation}
holds for all $f\in\mathscr{H}$. 

Using (\ref{eq:e3}), we see that every $f\in\mathscr{H}$ identifies
with a function on $\Omega$; and we shall use the same notation for
this function representations. By (\ref{eq:e12}), 
\begin{equation}
f=\sum_{n=1}^{\infty}\left\langle \psi_{n},f\right\rangle _{\mathscr{H}}\psi_{n}\label{eq:e13}
\end{equation}
holds for $f\in\mathscr{H}$, with norm-convergence. Hence (using
the function-representation):
\begin{eqnarray*}
f\left(t\right) & = & \sum_{n=1}^{\infty}\left\langle \psi_{n},f\right\rangle _{\mathscr{H}}\psi_{n}\left(t\right)\\
 & = & \left\langle \sum_{n=1}^{\infty}\overline{\psi_{n}\left(\cdot\right)}\psi_{n},f\right\rangle _{\mathscr{H}}\\
 & = & \left\langle \sum_{n=1}^{\infty}\overline{\left(\left(L^{*}L\right)^{-\frac{1}{2}}\varphi_{n}\right)}\left(\cdot\right)\left(\left(L^{*}L\right)^{-\frac{1}{2}}\varphi_{n}\right)\left(t\right),f\right\rangle _{\mathscr{H}}\\
 & = & \left\langle \left\langle T_{G}^{-\frac{1}{2}}l\left(\cdot\right),T_{G}^{-\frac{1}{2}}l\left(t\right)\right\rangle _{2},f\right\rangle _{\mathscr{H}}\\
 & = & \left\langle K^{G}\left(\cdot,t\right),f\right\rangle _{\mathscr{H}},
\end{eqnarray*}
valid for $\forall t\in\Omega$, and all $f\in\mathscr{H}$. This
is the desired RKHS-property for $\mathscr{H}$. 
\end{proof}

\section{Applications}

\subsection{Hilbert matrix}

In Theorem \ref{thm:rkhs} we assume that the given Hilbert space
$\mathscr{H}$ has a frame $\left\{ \varphi_{n}\right\} \subset\mathscr{K}$
consisting of functions on a set $\Omega$. So this entails both a
lower, and an upper bound, i.e., $0<B_{1}\leq B_{2}<\infty$. The
following example shows that the conclusion in the theorem is false
if there is \emph{not} a positive lower frame bound.

Set $\mathscr{H}=L^{2}\left(0,1\right)$, $\Omega=\left(0,1\right)$,
the open unit-interval, and $\varphi_{n}\left(t\right)=t^{n}$, $n\in\left\{ 0\right\} \cup\mathbb{N}=\mathbb{N}_{0}$.
In this case, the Gramian 
\begin{equation}
G_{n,m}=\int_{0}^{1}x^{n+m}dx=\frac{1}{n+m+1}
\end{equation}
is the $\infty\times\infty$ Hilbert matrix, see \cite{Ros58,Kat57,Tau49}.
In this case it is known that there is an upper frame bound $B_{2}=\pi$,
i.e., 
\begin{equation}
\sum_{n=0}^{\infty}\left|\int_{0}^{1}f\left(x\right)x^{n}dx\right|^{2}\leq\pi\int_{0}^{1}\left|f\left(x\right)\right|^{2}dx;
\end{equation}
in fact, for the operator-norm, we have $\left\Vert G\right\Vert _{l^{2}\rightarrow l^{2}}=\pi$. 

Moreover, $G$ defines a selfadjoint operator in $l^{2}\left(\mathbb{N}_{0}\right)$
with spectrum $\left[0,\pi\right]$, the closed interval. This implies
that there cannot be a positive lower frame-bound.

Further, it is immediate by inspection that $\mathscr{H}=L^{2}\left(0,1\right)$
is \emph{not} a RKHS.

\subsection{Random fields}
\begin{defn}[see \cite{MR2509983}]
Let $V$ be a countable discrete set (for example the vertex-set
in an infinite graph \cite{MR3195830,MR3203830}.) Let $\left(\Omega,\mathscr{F},P\right)$
be a probability space, and let $\mathscr{H}:=L^{2}\left(\Omega,P\right)$
be the corresponding Hilbert space of $L^{2}$-random variables, i.e.,
$f:\Omega\rightarrow\mathbb{R}$, measurable s.t. 
\begin{equation}
\mathbb{E}(\left|f\right|^{2})=\int_{\Omega}\left|f\right|^{2}dP<\infty.
\end{equation}
A system $\left(\varphi_{x}\right)_{x\in V}\subset\mathscr{H}$ is
said to be a \emph{random frame} if
\begin{equation}
\sum_{y\in V}\left|\mathbb{E}\left(\varphi_{x}\varphi_{y}\right)\right|^{2}<\infty,\quad\forall x\in V;\label{eq:r1}
\end{equation}
and the following implication holds: 
\begin{equation}
\Big[\sum_{y\in V}\mathbb{E}\left(\varphi_{x}\varphi_{y}\right)c_{y}=-c_{x},\;\left(c_{x}\right)\in l^{2}\left(V\right)\Big]\Longrightarrow\left(c_{x}\right)=0\;\mbox{in}\; l^{2}\left(V\right).\label{eq:r2}
\end{equation}
\end{defn}
\begin{cor}
Let $\left(\Omega,\mathscr{F},P\right)$ be a probability space and
let $\left(\varphi_{x}\right)_{x\in V}$ be a random frame in $L^{2}\left(\Omega,P\right)$.
Then 
\begin{enumerate}
\item \label{enu:r1}the \uline{generalized frame operator}
\begin{equation}
L^{*}L_{1}f=\sum_{x\in V}\mathbb{E}\left(\varphi_{x}f\right)\varphi_{x}\label{eq:rr1}
\end{equation}
is selfadjoint in $L^{2}\left(\Omega,P\right)$, generally unbounded;
and 
\item Let $a,b\in\mathbb{R}_{+}$ satisfying $0<a\leq b<\infty$, set $\alpha=\left[a,b\right]$.
Let $P_{E}\left(\cdot\right)$ be the projection valued measure of
$L^{*}L_{1}$, and $\mathscr{H}\left(\alpha\right)=P_{E}\left(\alpha\right)\mathscr{H}$,
then the estimate
\begin{equation}
a\:\mathbb{E}(\left|f\right|^{2})\leq\sum_{x\in V}\left|\mathbb{E}\left(\varphi_{x}f\right)\right|^{2}\leq b\:\mathbb{E}(\left|f\right|^{2})\label{eq:rr2}
\end{equation}
holds for all $f\in\mathscr{H}\left(\alpha\right)$. (i.e., If $P_{E}\left(\cdot\right)$
denotes the spectral resolution of the operator $L^{*}L_{1}$ in (\ref{enu:r1}),
we may take $\mathscr{H}\left(\alpha\right)=P_{E}\left(\alpha\right)\mathscr{H}\left(=P_{E}\left(\left[a,b\right]\right)\mathscr{H}\right)$
in (\ref{eq:rr2}).)
\end{enumerate}
\end{cor}
\begin{acknowledgement*}
The co-authors thank the following colleagues for helpful and enlightening
discussions: Professors Daniel Alpay, Sergii Bezuglyi, Ilwoo Cho,
Ka Sing Lau, Paul Muhly, Myung-Sin Song, Wayne Polyzou, Keri Kornelson,
and members in the Math Physics seminar at the University of Iowa.

\bibliographystyle{amsalpha}
\bibliography{ref}
\end{acknowledgement*}

\end{document}